\documentclass{amsart}
\usepackage{amssymb}
\usepackage{amscd}
\usepackage{multicol}
\usepackage{tikz}
\usepackage{graphics}
\usepackage{shuffle}
\usepackage[OT2,T1]{fontenc}
\usepackage[all]{xy}
\usepackage{extarrows}
\usepackage{comment}
\DeclareSymbolFont{cyrletters}{OT2}{wncyr}{m}{n}
\DeclareMathSymbol{\Sha}{\mathalpha}{cyrletters}{"58}

\usepackage{ulem}

\newtheorem{thm}{Theorem}[section]
\newtheorem{lem}[thm]{Lemma}
\newtheorem{cor}[thm]{Corollary}
\newtheorem{prop}[thm]{Proposition}  

\theoremstyle{plain}


\numberwithin{equation}{section}
\theoremstyle{definition}
\newtheorem{definition}[thm]{Definition}
\newtheorem{rmk}[thm]{Remark}

\newtheorem{exa}[thm]{Example}

\newcommand{\re}{{\rm Re}}

\newcommand{\bfa}{\boldsymbol{a}}

\newcommand{\bfd}{\boldsymbol{d}}

\newcommand{\bfk}{\boldsymbol{k}}
\newcommand{\bfl}{\boldsymbol{l}}
\newcommand{\bfm}{\boldsymbol{m}}
\newcommand{\bfn}{\boldsymbol{n}}

\newcommand{\bfs}{\boldsymbol{s}}

\newcommand{\bfepsilon}{\boldsymbol{\epsilon}}

\newcommand{\calD}{\mathcal{D}}

\newcommand{\N}{\mathbb N}
\newcommand{\Z}{\mathbb Z}
\newcommand{\Q}{\mathbb Q}
\newcommand{\R}{\mathbb R}
\newcommand{\C}{\mathbb C}

\title[Stirling polynomials and values of MZ(S)Fs at non-positive integers]
{Stirling polynomials and multiple zeta (star) functions at non-positive integers}
\author{Yoni Ishii}
\author{Takeshi Shinohara}
\address{(T. Shinohara) Graduate School of Mathematics, Nagoya University, Furo-cho, Chikusa-ku, Nagoya 464-8602 Japan}
\email{m21022w@math.nagoya-u.ac.jp}
\date{\today}

\begin{document}
\bibliographystyle{amsalpha+}

\begin{abstract}      
It is known that the values of multiple zeta functions (MZFs) at non-positive integers can be expressed by Bernoulli numbers. 
This paper gives explicit formulas for the values of MZFs and multiple zeta star functions (MZSFs) at non-positive integers using Stirling polynomials. 
We also study the following two points: a connection between MZFs and MZSFs at non-positive integers and a connection between reverse values and generalized Gregory coefficients studied by Matsusaka, Murahara, and Onozuka. 
\end{abstract}
\maketitle
\tableofcontents

\section{Introduction}\label{sec: Introduction}
Let $\N$, $\Z$, $\Q$, $\R$, $\C$ be the set of positive integers, integers, rational numbers, real numbers, and complex numbers, respectively. Set $\N_0:=\{0\}\cup\N$ to be the set of all non-negative integers. 
The {\it Euler-Zagier multiple zeta function} (MZF for short) $\zeta_r(s_1,\dots,s_r)$ is defined for $r\in\N$
and complex variables $\bfs=(s_1,\dots,s_r)\in\C^r$ by
{\small
\begin{align*}
 &\zeta_r(s_1,\dots,s_r)
  = \sum_{0<n_1<\cdots<n_r}\frac{1}{n_1^{s_1}\cdots n_r^{s_r}}
  = \sum_{m_1,\dots,m_r\in\N}\frac{1}{m_1^{s_1}\cdots (m_1+\cdots+m_r)^{s_r}}. 
\end{align*}}
The MZF $\zeta_r(s_1,\dots,s_r)$ converges absolutely in
\begin{align*}
 \calD_r:=\{ (s_1,\dots,s_r)\in\C^r \,|\,\re(s_{r-k+1}+\cdots+s_r)>k,\, k=1,\dots,r \}.
\end{align*}

In early 2000, Zhao \cite{z} and Akiyama, Egami, and Tanigawa \cite{aet} independently proved that MZFs have a meromorphic continuation to $\C^r$.
In particular, in \cite{aet}, they determined the set of all singularities of the MZF as follows.
\begin{align*}
 &s_r=1,\ s_{r-1}+s_r=2,1,0,-2,-4,\dots,\\
 &\quad\sum_{i=1}^k s_{r-i+1}\in\Z_{\le k}, (k=3,\dots,r).
\end{align*}
This tells us all {\it non-positive integer points} (the $r$-ple of non-positive integers $(l_1,\dots,l_r)\in\Z_{\le0}^r$ for some $r\in\N$) belong to the above singularities, 
except for the cases $r=1$ and $(-l_1,-l_2)\in\Z_{\le0}^2$ where $l_1+l_2$ is odd.
It is also shown that these points are indeterminacy, and hence the values of MZFs at non-positive integers depend on the limiting process.  

The problem of determining `nice' values of MZFs at non-positive integers is fundamental,
and many mathematicians have tackled this subject, 
\cite{aet}, \cite{at}, \cite{ka}, \cite{gz}, \cite{s09}, \cite{ko}, \cite{o}, \cite{ems}, \cite{fkmt}, \cite{em20}, 
\cite{s23}, for instance.
In \cite{aet}, the authors defined {\it regular/reverse values of MZFs} at non-positive integers: 
for $\bfl=(l_1,\dots,l_r)\in\N_0^r$, 
the regular values (reverse values, resp.) of MZFs are defined by
\begin{align*}
 &\zeta_r^{\rm reg}(-\bfl) = \zeta_r^{\rm reg}(-l_1,\dots,-l_r)
 := \lim_{s_1\rightarrow -l_1}\cdots \lim_{s_r\rightarrow -l_r}\zeta_r(s_1,\dots,s_r),\\
 &\zeta_r^{\rm rev}(-\bfl) = \zeta_r^{\rm rev}(-l_1,\dots,-l_r)
 := \lim_{s_r\rightarrow -l_r}\cdots \lim_{s_1\rightarrow -l_1}\zeta_r(s_1,\dots,s_r).
\end{align*}
By the recurrence relations of the MZFs (see Lemma \ref{lem: reccurence relation for MZF, regular type}, 
\ref{lem: reccurence relation for MZF, reverse type}), one finds that regular/reverse values can be expressed in terms of Bernoulli numbers, in principle.
However, Akiyama and Tangiawa \cite{at} gave explicit formulas
of the regular/reverse values of MZFs for special cases (cf. \cite[Theorem 2, 4]{at}):
let $l\in\N_0$. It follows
\begin{equation} \label{eqn: Akiyama, Tanigawa's results}
\begin{split}
 \zeta_r^{\rm reg}(-l,0,\dots,0) &= -\frac{1}{r!}\sum_{k=1}^r(-1)^kks(r,k)\zeta(-l-k+1), \\
 \zeta_r^{\rm rev}(0,\dots,0,-l) &= \frac{1}{(r-1)!}\sum_{k=1}^rs(r,k)\zeta(-l-k+1).
\end{split}
\end{equation} 
where $s(n,m)$ ($n,m\in\N_0$) is the Stirling number defined in \S\ref{sec: Preliminaries}.

It is known that values of various MZFs considered in previous works (\cite{ko}, \cite{ems}, \cite{fkmt}, and \cite{em20}) can be expressed by Bernoulli numbers, except for \eqref{eqn: Akiyama, Tanigawa's results}. 
This paper aims to generalize \eqref{eqn: Akiyama, Tanigawa's results} for any indices, 
namely our first main result is the following.

\bigskip\noindent
{\bf Theorem \ref{thm: explicit values of MZFs at non-positive integer points using Stirling numbers}.} 
{\it 
Let $\bfl=(l_1,\dots,l_r)\in\N_0^r$ and $s\in\C$. We have
\begin{align*}
 \zeta_r^{{\rm rev}}(-\bfl)
 = \sum_{0\le k_i \le l_i \atop 1\le i\le r} 
    \prod_{j=1}^r S(l_j,k_j,L_{j-1}+j)
    \frac{(K_j+j-1)!}{(K_{j-1}+j-1)!}\,
    \zeta_{r+K_r}^{{\rm rev}}(\mathbf{0}),
\end{align*}
where $L_j:=l_1+\cdots+l_j$, $K_j:=k_1+\cdots+k_j$ ($1\le j \le r$), $L_0=K_0:=0$ and $S(l_j,k_j,L_{j-1}+j)$ is the Stirling polynomial of the second kind defined by \eqref{eqn: Stirling polynomial of the second kind}.}
\bigskip

In \S\ref{subsec: MZF vs MZSF}, 
we will also prove this type of formula for {\it multiple zeta star functions} (MZSFs) defined by
\begin{align*}
 \zeta_r^{\star}(s_1,\dots,s_r) = \sum_{\circ=\{, \text{ or } +\}} \zeta(s_1\circ\cdots\circ s_r).
\end{align*}
For details of MZSFs, see \S\ref{subsec: MZF vs MZSF}.
After that, we will show our second main theorem.

\bigskip\noindent
{\bf Theorem \ref{thm: MZF vs MZSF at non-positive integers}.} 
{\it 
Let $r\in\N$ and $\bfl=(l_1,\dots,l_r)\in\N_0^r$ with $l_1\ge1$. 
We have 
\begin{align*}
&\zeta_r^{\rm reg}(-\bfl) = (-1)^{r+|\bfl|} \zeta_r^{\star \,{\rm reg}}(-\bfl), \\
&\zeta_r^{\rm rev}(-\bfl) = (-1)^{r+|\bfl|} \zeta_r^{\star \,{\rm rev}}(-\bfl)
\end{align*}
where $|\bfl|:=l_1+\cdots+l_r$.}
\bigskip

In the last part of this paper, we study a connection between the reverse values 
and {\it asymptotic coefficients of MZFs} (for definition, see \S\ref{subsec: asym. coeff.}).
More precisely, we will prove the following as an application of Theorem \ref{thm: explicit values of MZFs at non-positive integer points using Stirling numbers}.

\bigskip\noindent
{\bf Theorem \ref{thm: reverse values vs Gregory coefficients}.} 
{\it Let $\bfl=(l_1,\dots,l_{r})\in\N_0^{r}$.
It follows that  
\begin{equation}
 \zeta_{r}^{\rm rev}(-\bfl) 
 = \sum_{0\le k_i \le l_i \atop 1\le i\le r} 
    c(\bfk,\bfl)
   \sum_{j=0}^{\lfloor\frac{r-1}{2}\rfloor}
   \sum_{k=0}^{r-1-2j}\sum_{(\bfm,\bfn)\in J(j,k)}
   \prod_{p=1}^{j+1}G_{m_p,n_p-m_p+2}
\end{equation}
where $G_{m,n}$ is the generalized Gregory coefficient. For other notations, see \S\ref{subsec: asym. coeff.}. }
\bigskip

The present paper is organized as follows.
In \S\ref{sec: Preliminaries}, 
we introduce Stirling polynomials of the first/second kind and show their properties. 
Those combinatorial objects were already studied in \cite{bm} but our definitions are slightly different. 
After that, we briefly check the basic properties of multiple Hurwitz zeta functions.
In \S\ref{sec: Main results}, we will prove our main theorems.
Specifically, in \S\ref{subsec: Stirling no.s and MZFs}, we prove Theorem \ref{thm: explicit values of MZFs at non-positive integer points using Stirling numbers}. 
In \S\ref{subsec: MZF vs MZSF}, we briefly explain basic properties of MZSFs and state that MZSFs have similar explicit formula at non-positive integers. After that we will prove Theorem \ref{thm: MZF vs MZSF at non-positive integers}.
In \S\ref{subsec: asym. coeff.}, we review Onozuka's work to consider asymptotic coefficients of MZFs. And then, we give a proof of Theorem \ref{thm: reverse values vs Gregory coefficients}.

\bigskip
{\it Acknowledgments.} T.S. has been supported by grants JSPS KAKENHI JP24KJ1252.

\section{Preliminaries}\label{sec: Preliminaries}
We consider Stirling polynomials of the first/second kind and show some fundamental properties.
We also review the multiple Hurwitz zeta functions that will be used in \S\ref{sec: Main results}.  

\subsection{Stirling polynomials}\label{subsec: Stirling polynomials}
In this subsection, we consider Stirling polynomials of the first/second kind studied in \cite{bm}, but our definition is slightly modified.
First, we recall the Stirling numbers. 
There are many ways to define Stirling numbers, 
but here we adopt the following definition.
Let $n,m\in\N_0$.
The Stirling number of the first kind $s(n,m)\in\Z$ is defined by
\begin{align*}
 (X)_n:=X(X-1)\cdots(X-n+1)=\sum_{m=0}^n s(n,m)X^m\in\Z[X].
\end{align*}
We set $(X)_0:=1$ and $s(0,0)=1$.
The Stirling number of the second kind $S(n,m)\in\Z$ is defined by
\begin{align*}
 X^n=\sum_{m=0}^nS(n,m)(X)_m\in\Z[X]
\end{align*}
and $S(0,0)=1$. We understand $s(n,m)=S(n,m)=0$ for $n<m$.

As a natural generalization of the above, we introduce Stirling polynomials. 

\begin{definition}{(cf. \cite{bm})}
Let $n,m\in\N_0$. 
We define the {\it Stirling polynomial of the first kind} $s(n,m,Y)\in\Z[Y]$ by
\begin{align}\label{eqn: Stirling polynomial of the first kind}
 (X-Y)_n = \sum_{m=0}^n s(n,m,Y)X^m \in \Z[X,Y] 
\end{align}
and $s(0,0,Y)=1$.
Similarly, we define the {\it Stirling polynomial of the second kind} $S(n,m,Y)\in\Z[Y]$ by 
\begin{align}\label{eqn: Stirling polynomial of the second kind}
 (X+Y)^n =\sum_{m=0}^nS(n,m,Y)(X)_m \in \Z[X,Y]  
\end{align}
and $S(0,0,Y)=1$. We put $s(n,m,Y)=S(n,m,Y)=0$ for $n< m$.
\end{definition}

Note that $s(n,m,0)=s(n,m)$ and $S(n,m,0)=S(n,m)$ for $n,m\in\N_{\ge0}$.
B{\'e}nyi and Matsusaka defined Stirling polynomials of the first/second kind 
$ \begin{bmatrix} n\\m \end{bmatrix}_Y, \ \left\{\begin{matrix} n\\m \end{matrix}\right\}_Y$ in the combinatorial context.
By \cite[Proposition 2.21]{bm}, one finds
\begin{align*}
 \begin{bmatrix} n \\m \end{bmatrix}_Y = (-1)^{n-m}s(n,m,Y).
\end{align*}
Moreover, it is known that $\left\{\begin{matrix} n\\m \end{matrix}\right\}_Y$ has the following generating function
(cf. \cite{b})
\begin{align*}
 \frac{(e^X-1)^m}{m!}e^{YX} = \sum_{n\ge0}\left\{\begin{matrix} n\\m \end{matrix}\right\}_Y\frac{X^n}{n!}.
\end{align*}
We will prove that $S(n,m,X)$ has the same generating function 
(see Lemma \ref{lem: gen. func. of Stirling poly. of the 2nd kind}) and therefore we have
\begin{align*}
 \left\{\begin{matrix} n\\m \end{matrix}\right\}_Y = S(n,m,Y).
\end{align*}

\begin{rmk} 
Let $n,m\in\N_0$ with $n\ge m$. By definition, we have
\begin{align}
 s(n,m,X) &= \sum_{k=0}^{n-m} \binom{m+k}{m}s(n,m+k)(-X)^k,\\
 S(n,m,X) &= \sum_{k=0}^{n-m} \binom{n}{n-k}S(n-k,m)X^k. \label{eqn: Stirling poly. vs number of the second kind}
\end{align}
\end{rmk} 

\begin{lem}\label{lem: Stirling polynomial transform}
For all $n,m\in\N_0$, we have 
\begin{align*}
 \sum_{k=m}^ns(n,k,X)S(k,m,X) = \sum_{k=m}^nS(n,k,X)s(k,m,X) = \delta_{nm},
\end{align*}
where $\delta_{nm}$ is the Kronecker's delta.
\end{lem}

\begin{proof}
Since one can prove in similar way, we only prove second equality:
\begin{align*}
 \sum_{k=m}^nS(n,k,X)s(k,m,X) = \delta_{nm}.
\end{align*}
By \eqref{eqn: Stirling polynomial of the second kind} and change of variable $X\mapsto X-Y$, we have 
\begin{align*}
 X^n=\sum_{m=0}^n S(n,m,Y)(X-Y)_m.
\end{align*}
Then, by \eqref{eqn: Stirling polynomial of the first kind}, one has 
\begin{align*}
 X^n=\sum_{m=0}^n S(n,m,Y)\sum_{l=0}^ms(m,k,Y)X^l 
    = \sum_{l=0}^n\left( \sum_{m=l}^nS(n,m,Y)s(m,l,Y) \right)X^l.
\end{align*}
Comparing the coefficients on both sides, we obtain the claim.
\end{proof}

\begin{rmk}
By Lemma \ref{lem: Stirling polynomial transform}, we have the following Stirling polynomial transform:
\begin{align}\label{eqn: Stirling polynomial transform}
 a_n=\sum_{k=0}^{n}s(n,k,X)b_k\quad \Leftrightarrow\quad b_n=\sum_{k=0}^{n}S(n,k,X)a_k
\end{align}
for $n\in\N_0$ and any sequences $(a_n)_{n\in\N_0}$, $(b_n)_{n\in\N_0}\subset\C$.
\end{rmk}

It is well known that the Stirling number of the second kind has the following generating function.
\begin{align*}
 \frac{(e^X-1)^k}{k!} = \sum_{n\ge0}S(n,k)\frac{X^n}{n!} \quad (k\in\N_0).
\end{align*}
 
The Stirling polynomials of the second kind also have the generating function.

\begin{lem}\label{lem: gen. func. of Stirling poly. of the 2nd kind}
For any $k\in\N_0$, we have 
\begin{align*}
 \frac{(e^X-1)^k}{k!}e^{YX} = \sum_{n\ge0}S(n,k,Y)\frac{X^n}{n!}
\end{align*}
\end{lem}

\begin{proof}
By \eqref{eqn: Stirling poly. vs number of the second kind}, we have
\begin{align*}
\sum_{n\ge0}S(n,k,Y)\frac{X^n}{n!} 
&= \sum_{n\ge0}\left( \sum_{l\ge0} \binom{n}{n-l}S(n-l,k)Y^l \right)\frac{X^n}{n!} \\
&= \sum_{l\ge0}\frac{(YX)^l}{l!}\left( \sum_{n\ge0} S(n-l,k)\frac{X^{n-l}}{(n-l)!} \right)
= \frac{(e^X-1)^k}{k!}e^{YX}.
\end{align*}
Note that $S(n,m)=0$ when $n<m$.
\end{proof}

Differentiating (with respect to $X$) both sides of the generating function for the Stirling polynomial of the second kind, 
we obtain the following recurrence relation.

\begin{lem}\label{lem: rec. rel. for Stirling poly. of the 2nd kind}
Let $n,m\in\N_0$. We have 
\begin{align*}
 YS(n,m,Y) = S(n+1,m,Y) - S(n,m-1,Y+1).
\end{align*}
\end{lem}

By combining Lemma \ref{lem: rec. rel. for Stirling poly. of the 2nd kind} and 
the recurrence relation for Stirling numbers of the second kind
\footnote{It is well known that $S(n,m) = mS(n-1,m)+S(n-1,m-1)$.}, one can prove the following.

\begin{lem}
Let $n,k\in\N_0$. For any $m\in\N_0$ with $m\le n$, it follows
\begin{align*}
 S(n,k,X) = \sum_{i=0}^m S(m,i,X+k-i)S(n-m,k-i,X).
\end{align*}
\end{lem}

\subsection{Choi's Hurwitz zeta function}
We review the {\it multiple Hurwitz zeta function} studied by Choi \cite{c}. 
Choi considered the following function as a special case of the Barnes zeta function.
Let $r\in\N$, $s, z\in\C$ with $\re(s)>r$ and $\re(z)>0$. 
We define
\begin{align*}
 \zeta(r;s;z):= \sum_{m_1,\dots,m_r\in\N_0}\frac{1}{(m_1+\cdots+m_r+z)^s}.
\end{align*}
When $r=1$, $\zeta(1;s;z)$ coincides with the classical Hurwitz zeta function
\begin{align*}
 \zeta(1;s;z):=\sum_{n=0}^{\infty} \frac{1}{(n+z)^s}.
\end{align*}
Choi called the function $\zeta(r;s;z)$ multiple Hurwitz zeta function, 
but here we call $\zeta(r;s;z)$ Choi's Hurwitz zeta function.
We also note that the following holds
\begin{equation}\label{eqn: mHzf vs mzf}
 \zeta(r;s;r) = \sum_{m_1,\dots,m_r\in\N} \frac{1}{(m_1+\cdots+m_r)^s} = \zeta_r(0,\dots,0,s).
\end{equation}
Since $\re(s)>r$, $\zeta_r(0,\dots,0,s)$ converges.

Choi proved that $\zeta(r;s;z)$ has an analytic continuation to the whole $s$-plane 
and calculated explicit values at non-positive integers as follows:
\begin{lem}{(\cite{c})}
Let $r\in\N$ and $z\in\C$ with $\re(z)>0$. Choi's Hurwitz zeta function $\zeta(r;s;z)$ has a meromorphic continuation to the whole complex plane $s\in\C$ with singularities $s=1,\dots,r$. 
Moreover, its values at non-positive integers can be expressed explicitly as follows. 
\begin{align*}
 \zeta(r;-l;z) = (-1)^r\frac{l!}{(r+l)!}B_{r+l}^{(r)}(z) \quad (l\in\N_0)
\end{align*}
where $B_{n}^{(m)}$ is the $n$-th Bernoulli polynomial of order $m$ defined by 
\begin{align}\label{eqn: definition of Bernoulli polynomial of order m}
 \frac{X^me^{zX}}{(e^X-1)^m} = \sum_{n\ge0}B_n^{(m)}(z)\frac{X^n}{n!}.
\end{align}
\end{lem}

In \cite{c}, he also pointed out that the multiple Hurwitz zeta function $\zeta(r;s;z)$ can be expressed as a simple series:
\begin{equation}\label{eqn: simple series expression of mHzf}
 \zeta(r;s;z) = \sum_{n=0}^{\infty} \binom{r+n-1}{n}(n+z)^{-s}.
\end{equation}
Thus, comparing \eqref{eqn: mHzf vs mzf} and \eqref{eqn: simple series expression of mHzf}, 
we have the following relation.
\begin{equation}\label{eqn: mzf has a simple series exp}
 \zeta_r(0,\dots,0,s) = \sum_{n=0}^{\infty}\binom{r+n-1}{n}(n+r)^{-s}.
\end{equation}

Here, we prove a combinatorial property for $\zeta(r;s;z)$.

\begin{prop}
Let $r\in\N$, $m\in\N_0$, and $z\in\C$ with $\re(z)>0$. We have 
\begin{equation}\label{eqn: comb. rel. for mHzf}
 \zeta(r;s;z) = \sum_{k=0}^m \binom{m}{k}\zeta(r-m+k;s;z+k)
\end{equation}
for all $s\in\C$ except for the singularities.
\end{prop}

\begin{proof}
We prove the claim by induction on $m$.
We assume $\re(s)>r$. When $m=1$, by \eqref{eqn: simple series expression of mHzf}, we have 
{\small
\begin{align*}
\zeta(r-1;s;z)+\zeta(r;s;z+1) 
&= \sum_{n=0}^{\infty}\binom{r+n-2}{n}(n+z)^{-s} + \sum_{n=0}^{\infty}\binom{r+n-1}{n}(n+z+1)^{-s} \\
&= z^{-s} + \sum_{n=0}^{\infty} \left\{ \binom{r+n-1}{n+1} + \binom{r+n-1}{n} \right\}(n+z+1)^{-s}.
\end{align*}}
Since $\binom{r+n-1}{n+1} + \binom{r+n-1}{n}=\binom{r+n}{n+1}$, we get the claim for $m=1$. 
We assume \eqref{eqn: comb. rel. for mHzf} for $m-1$.
We calculate the right-hand side of \eqref{eqn: comb. rel. for mHzf}.
\begin{align*}
&\sum_{k=0}^m \binom{m}{k}\zeta(r-m+k;s;z+k) \\
&= \sum_{k=0}^{m-1} \binom{m-1}{k}\zeta(r-m+k;s;z+k)+\sum_{k=1}^m \binom{m}{k-1}\zeta(r-m+k;s;z+k) \\
&= \zeta(r-1;s;z) + \sum_{k=0}^{m-1} \binom{m-1}{k}\zeta(r-m+k+1;s;z+k+1) \\
&= \zeta(r-1;s;z) + \zeta(r;s;z+1) = \zeta(r;s;z). 
\end{align*}
In the second and third equality, we used the induction hypothesis. Thus, induction is completed.

We know that all members on both sides of \eqref{eqn: comb. rel. for mHzf} have analytic continuations
to the whole complex plane, \eqref{eqn: comb. rel. for mHzf} holds for all $s\in\C$ except for singularities.
\end{proof}

\section{Main results}\label{sec: Main results}
In this section, we prove our main theorems stated in \S\ref{sec: Introduction}.

\subsection{Stirling polynomials and MZFs at non-positive integer points}\label{subsec: Stirling no.s and MZFs}
We first show the following proposition, which is a key to prove 
Theorem \ref{thm: explicit values of MZFs at non-positive integer points using Stirling numbers}.

\begin{prop}\label{prop: reverse function with arbitrary non-positive integers vs zeros}
Let $\bfl=(l_1,\dots,l_r)\in\N_0^r$ and $s\in\C\setminus\{1,\dots,r+L_r\}$. We have
{\small
\begin{align}\label{eqn: reverse function with arbitrary non-positive integers vs zeros}
 \sum_{0\le k_i \le l_i \atop 1\le i\le r}
  \prod_{j=1}^r s(l_j,k_j,L_{j-1}+j)\cdot 
  \zeta_r(-\bfl',-l_r+s)
 = \prod_{j=1}^r
    \frac{(L_j+j-1)!}{(L_{j-1}+j-1)!} \cdot
    \zeta_{r+L_r}(\mathbf{0},s),
\end{align}}
where, $-\bfl'=(-l_1,\dots,-l_{r-1})$, $L_j:=l_1+\cdots+l_j$ ($1\le j \le r$), $L_0:=0$ 
and $s(l_j,k_j,L_{j-1}+j)$ is the Stirling polynomial defined 
by \eqref{eqn: Stirling polynomial of the first kind}.
\end{prop}

\begin{proof}
Let $M_i:=\sum_{k=1}^im_k$. 
First we assume $\re(s)>r+L_r$. 
Since $(-l_1,\dots,-l_{r-1},-l_r+s)\in\calD_r$ 
and equation \eqref{eqn: Stirling polynomial of the first kind}, we have
{\small
\begin{align*}
&\sum_{0\le k_i \le l_i \atop 1\le i\le r}
  \prod_{j=1}^r s(l_j,k_j,L_{j-1}+j)\cdot 
  \zeta_r(-l_1,\dots,-l_{r-1},-l_r+s) \\
&= \sum_{m_j\in\N \atop 1\le j\le r} 
    \frac{ (\sum_{k_1=0}^{l_1}s(l_1,k_1,L_{0}+1)M_1^{l_1})\cdots(\sum_{k_1=0}^{l_1}s(l_r,k_r,L_{r-1}+r)M_r^{l_r}) }
         {M_r^s} \\
&= \sum_{m_j\in\N \atop 1\le j\le r} 
   \frac{(M_1-L_0-1)_{l_1}\cdots(M_r-L_{r-1}-r)_{l_r}}{M_r^s}
 = \sum_{m_j\in\N_0 \atop 1\le j\le r} 
   \frac{(M_1-L_0)_{l_1}\cdots(M_r-L_{r-1})_{l_r}}{(M_r+r)^s}.
\end{align*}}
Note that $(M_j-L_{j-1})_{l_j}=0$ when $L_{j-1}\le M_j\le L_j-1$ ($j=1,\dots,r$). 
Thus, we have 
{\small
\begin{align*}
\sum_{m_j\in\N_0 \atop 1\le j\le r} 
 \frac{(M_1-L_0)_{l_1}\cdots(M_r-L_{r-1})_{l_r}}{(M_r+r)^s}
= \sum_{n=0}^{\infty}
   (n+r+L_{r})^{-s}\cdot
   \sum_{\substack{M_r=n+L_r, \\ m_j\ge0 \\ 1\le j\le r}}
    \prod_{j=1}^r \frac{(M_j-L_{j-1})!}{(M_j-L_j)!}.
\end{align*}}
By \eqref{eqn: mzf has a simple series exp}, we have 
\begin{align*}
 \zeta_{r+L_r}(0,\dots,0,s)
= \sum_{n=0}^{\infty}\binom{r+L_r+n-1}{n}(n+r+L_{r})^{-s},
\end{align*}
so it is enough to show the following.
\begin{align}\label{eqn: comparing coefficients of the both side in the claim}
\sum_{\substack{M_r=n+L_r, \\ m_j\ge0 \\ 1\le j\le r}}
 \prod_{j=1}^r \frac{(M_j-L_{j-1})!}{(M_j-L_j)!}
= \prod_{j=1}^r
   \frac{(L_j+j-1)!}{(L_{j-1}+j-1)!} \cdot
   \binom{r+L_r+n-1}{n}.
\end{align}
We prove equation \eqref{eqn: comparing coefficients of the both side in the claim} 
by induction on $r\in\N$ and $n\in\N_0$. 
When $r=1$, equation \eqref{eqn: comparing coefficients of the both side in the claim} holds for any $n\in\N_0$.

We assume that equation \eqref{eqn: comparing coefficients of the both side in the claim} 
holds for $r-1$ and all $n\in\N_0$.
Then, we have
\begin{align*}
&\sum_{\substack{M_r=n+L_r, \\ m_j\ge0 \\ 1\le j\le r}}
 \prod_{j=1}^r \frac{(M_j-L_{j-1})!}{(M_j-L_j)!}
= \sum_{m_r=0}^{n+L_r} 
    \frac{(n+l_r)!}{n!}
    \sum_{\substack{ M_{r-1}=n+L_r-m_r \\ m_j\ge0 \\ 1\le j\le r-1 }}
    \prod_{j=1}^{r-1} \frac{(M_j-L_{j-1})!}{(M_j-L_j)!} \\
&= \sum_{m_r=0}^{n+L_r} 
    \frac{(n+l_r)!}{n!}
    \prod_{j=1}^{r-1}
    \frac{(L_j+j-1)!}{(L_{j-1}+j-1)!} \cdot
    \binom{r-1+L_{r-1}+n+l_r-m_r-1}{n+l_r-m_r} \\ 
&= \frac{(n+l_r)!}{n!}
    \prod_{j=1}^{r-1}
    \frac{(L_j+j-1)!}{(L_{j-1}+j-1)!} \cdot 
    \binom{r-1+L_{r}+n}{n+l_r}.
\end{align*}
In the second equality, we use the induction hypothesis. 
The last equality holds because of the following.
\begin{align*}
\sum_{m_r=0}^{n+L_r}
 \binom{r-1+L_{r-1}+n+l_r-m_r-1}{n+l_r-m_r}
= \binom{r-1+L_{r}+n}{n+l_r}.
\end{align*}
Then, it is easy to show 
{\small
\begin{align*}
\frac{(n+l_r)!}{n!}
 \prod_{j=1}^{r-1}
 \frac{(L_j+j-1)!}{(L_{j-1}+j-1)!} \cdot 
 \binom{r-1+L_{r}+n}{n+l_r}
= \prod_{j=1}^r
   \frac{(L_j+j-1)!}{(L_{j-1}+j-1)!} \cdot
   \binom{r+L_r+n-1}{n}
\end{align*}}
for all $n\in\N_0$. This completes the induction.

Furthermore, both sides of \eqref{eqn: reverse function with arbitrary non-positive integers vs zeros} 
can be continued to the whole complex plane as a function of $s\in\C$ and, therefore, 
\eqref{eqn: reverse function with arbitrary non-positive integers vs zeros} holds for $s\in\C\setminus\{1,\dots,r+L_r\}$.
\end{proof}

\begin{cor}\label{cor: applying Stirling transform}
Let $\bfl=(l_1,\dots,l_r)\in\N_0^r$ and $s\in\C\setminus\{1,\dots,r+L_r\}$. We have
{\small
\begin{align}\label{eqn: main theorem, explicit values of MZFs using Stirling polynomials}
 \zeta_r(-\bfl',-l_r+s)
 = \sum_{0\le k_i \le l_i \atop 1\le i\le r} 
    \prod_{j=1}^r S(l_j,k_j,L_{j-1}+j)
    \frac{(K_j+j-1)!}{(K_{j-1}+j-1)!} \cdot
    \zeta_{r+K_r}(\mathbf{0},s),
\end{align}}
where $-\bfl'=(-l_1,\dots,-l_{r-1})$, $L_j:=l_1+\cdots+l_j$, $K_j:=k_1+\cdots+k_j$ ($1\le j \le r$), $L_0=K_0:=0$ 
and $S(l_j,k_j,L_{j-1}+j)$ is the Stirling polynomial defined by 
\eqref{eqn: Stirling polynomial of the second kind}.
\end{cor}

\begin{proof}
The claim can be obtained by applying the Stirling polynomial transform \eqref{eqn: Stirling polynomial transform} inductively.
\end{proof}

\begin{thm}\label{thm: explicit values of MZFs at non-positive integer points using Stirling numbers}
Let $l_1,\dots,l_r\in\N_0$. We have
\begin{align*}
\zeta^{\rm rev}_r(-\bfl)
= \sum_{0\le k_i \le l_i \atop 1\le i\le r} 
    \prod_{j=1}^r S(l_j,k_j,L_{j-1}+j)
    \frac{(K_j+j-1)!}{(K_{j-1}+j-1)!} \cdot
    \zeta^{\rm rev}_r(\boldsymbol{0}).
\end{align*}
\end{thm}

\begin{proof}
This is a direct consequence of Corollary \ref{cor: applying Stirling transform}.
\end{proof}

\subsection{MZF vs MZSF at non-positive integers}\label{subsec: MZF vs MZSF}
Next, we focus on the multiple zeta star functions (MZSFs).
The MZSF is defined by infinite series 
\begin{align*}
 \zeta_r^*(s_1,\dots,s_r) = \sum_{0<n_1\le n_2\le\cdots\le n_r}\frac{1}{n_1^{s_1}\cdots n_r^{s_r}}.
\end{align*}
The MZSF also converges in $\calD_r$ (see \S\ref{sec: Introduction} for definition) and in this region, we have 
\begin{align*}
 \zeta_r^{\star}(s_1,\dots,s_r) = \sum_{\circ=, \text{ or }+} \zeta(s_1\circ\cdots\circ s_r).
\end{align*}
Since MZFs have a meromorphic continuation to the whole complex space, one sees that MZSF can be continued to $\C^r$ with the following set of singularities: 
\begin{align*}
 &s_r=1,\ s_{r-1}+s_r=2,1,0,-2,-4,\dots,\\
 &\quad\sum_{i=1}^k s_{r-i+1}\in\Z_{\le k}, (k=3,\dots,r).
\end{align*}
So, we can also consider the regular/reverse values of MZSFs at non-positive integers as follows.
\begin{definition}
Let $\bfl=(l_1,\dots,l_r)\in\N_0^r$. we have
\begin{align*}
 &\zeta_r^{\star\,\rm reg}(-\bfl) = \zeta_r^{\star\,\rm reg}(-l_1,\dots,-l_r)
 := \lim_{s_1\rightarrow -l_1}\cdots \lim_{s_r\rightarrow -l_r}\zeta_r^{\star}(s_1,\dots,s_r),\\
 &\zeta_r^{\star\,\rm rev}(-\bfl) = \zeta_r^{\star\,\rm rev}(-l_1,\dots,-l_r)
 := \lim_{s_r\rightarrow -l_r}\cdots \lim_{s_1\rightarrow -l_1}\zeta_r^{\star}(s_1,\dots,s_r).
\end{align*}
\end{definition}

To give recurrence relations for MZSFs, we begin with the two types of recurrence relations for MZFs.
\begin{lem}{(cf. \cite{aet})}\label{lem: reccurence relation for MZF, regular type}
Let $l_r\in\N_0$. We have 
{\small
\begin{equation}
\label{eqn: rec. rel. for MZF regular type}
\zeta_r(\bfs,-l_r)
= \frac{-1}{l_r+1}\zeta_{r-1}(\bfs',s_{r-1}-l_r-1)
   + \sum_{k=0}^{l_r}\binom{l_r}{k}\zeta_{r-1}(\bfs',s_{r-1}-l_r+k)\zeta(-k) 
\end{equation}}
where $\bfs=(s_1,\dots,s_{r-1})\in\C^{r-1}$ and $\bfs'=(s_1,\dots,s_{r-2})$.
\end{lem}

\begin{lem}{(cf. \cite{aet})}\label{lem: reccurence relation for MZF, reverse type}
Let $l_1\in\N_0$. We have 
{\small
\begin{equation}\label{eqn: rec. rel. for MZF reverse type}
\begin{split}
\zeta_r(-l_1,\bfs)
= &\frac{1}{l_1+1}\zeta_{r-1}(-l_1+s_2-1,\bfs')
  - \sum_{k=0}^{l_1}\binom{l_1}{k}\zeta_{r-1}(-l_1+s_2+k,\bfs')\zeta(-k) \\
  &+\zeta(-l_1)\zeta_{r-1}(\bfs) - \zeta_{r-1}(-l_1+s_2,\bfs')
\end{split}
\end{equation}}
where $\bfs=(s_2,\dots,s_r)\in\C^{r-1}$ and $\bfs'=(s_3,\dots,s_r)$.
\end{lem}

Using \eqref{eqn: rec. rel. for MZF regular type} and \eqref{eqn: rec. rel. for MZF reverse type}, one can easily prove the recurrence relation for MZSFs.

\begin{lem}\label{lem: recurrence relation for MZSF}
Let $\mathbf{s}=(s_1,\dots,s_{r-1})\in\C^{r-1}$ and $l_r\in\N_0$. We have 
{\small
\begin{align*}
 \zeta_r^{\star}(\mathbf{s},-l_r)
 = -\frac{1}{l_r+1}\zeta_{r-1}^{\star}(\mathbf{s}',s_{r-1}-l_r-1)
   +\sum_{k=0}^{l_r}\binom{l_r}{k}\zeta_{r-1}^{\star}(\mathbf{s}',s_{r-1}-l_r+k)\zeta^{\star}(-k)
\end{align*}}
where $\zeta^{\star}(0):=1/2$ and $\zeta^{\star}(-k)=\zeta(-k)$ for $k\in\N$.
\end{lem}

\begin{proof}
By definition of MZSF, it follows 
\begin{equation}\label{eqn: explicit calculation of the rec. rel. for MZSF 1}
\begin{split}
\zeta_r^{\star}(\mathbf{s},-l_r)
 &=\sum_{\substack{\circ_j\in\{,\text{ or }+\} \\ 1\le j\le r-1}}
   \zeta(s_1\circ_1\dots\circ_{r-2}s_{r-1}\circ_{r-1}-l_r) \\
 &= \sum_{\substack{\circ_j\in\{,\text{ or }+\} \\ 1\le j\le r-2}}
    \Big\{ \zeta(s_1\circ_1\dots\circ_{r-2}s_{r-1},-l_r) + \zeta(s_1\circ_1\dots\circ_{r-2}s_{r-1}-l_r) \Big\} \\
 &= \sum_{\substack{\circ_j\in\{,\text{ or }+\} \\ 1\le j\le r-2}}
    \zeta(s_1\circ_1\dots\circ_{r-2}s_{r-1},-l_r) + \zeta_{r-1}^{\star}(\mathbf{s}',s_{r-1}-l_r).
\end{split}
\end{equation}
By Lemma \ref{lem: reccurence relation for MZF, regular type}, one has 
{\small
\begin{equation}\label{eqn: explicit calculation of the rec. rel. for MZSF 2}
\begin{split}
 \sum_{\substack{\circ_j\in\{,\text{ or }+\} \\ 1\le j\le r-2}}
 &\zeta(s_1\circ_1\dots\circ_{r-2}s_{r-1},-l_r) \\
 &\hspace{-13mm}= \!\!\!
   \sum_{\substack{\circ_j\in\{,\text{ or }+\} \\ 1\le j\le r-2}}\!\!
   \left\{ \frac{-1}{l_r+1} \zeta(s_1\circ_1\dots\circ_{r-2}s_{r-1}-l_r-1) 
           \!+ \!\sum_{k=0}^{l_r}\binom{l_r}{k} \zeta(s_1\circ_1\dots\circ_{r-2}s_{r-1}-l_r+k)\zeta(-k) 
   \right\}\\
 &\hspace{-13mm}= -\frac{1}{l_r+1} \zeta_{r-1}^{\star}(\mathbf{s}',s_{r-1}-l_r-1)
    + \sum_{k=0}^{l_r}\binom{l_r}{k} \zeta_{r-1}^{\star}(\mathbf{s}',s_{r-1}-l_r+k)\zeta(-k).
\end{split}
\end{equation}}
Note that the second equality holds because of the definition of MZSF.
Since $\zeta(0)=-1/2$, by \eqref{eqn: explicit calculation of the rec. rel. for MZSF 1}
and \eqref{eqn: explicit calculation of the rec. rel. for MZSF 2}, we obtain the claim. 
\end{proof}

We note that when $r=2$, one has 
\begin{align*}
 \zeta_2^{\star}(s_1,-l_2)
 = \frac{-1}{l_2+1}\zeta(s_1-l_2-1) + \sum_{k=0}^{l_2}\binom{l_2}{k}\zeta(s_1-l_2+k)\zeta^{\star}(-k).
\end{align*}
Similarly, one can prove the following.

\begin{lem}\label{lem: recurrence relation for MZSF, reverse type}
Let $\mathbf{s}=(s_2,\dots,s_{r})\in\C^{r-1}$ and $l_r\in\N_0$. We have 
{\small
\begin{align*}
 \zeta_r^{\star}(-l_1,\mathbf{s})
 = &\frac{1}{l_1+1}\zeta_{r-1}^{\star}(-l_1+s_2-1,\mathbf{s}')
   - \sum_{k=0}^{l_1}\binom{l_1}{k}\zeta_{r-1}^{\star}(-l_1+s_2+k,\mathbf{s}')\zeta^{\star}(-k) \\
   &+\zeta(-l_1)\zeta_{r-1}^{\star}(\bfs) - \zeta_{r-1}^{\star}(-l_1+s_2,\mathbf{s}')
\end{align*}}
where $\zeta^{\star}(0):=1/2$ and $\zeta^{\star}(-k)=\zeta(-k)$ for $k\in\N$.
\end{lem}

\begin{thm}\label{thm: explicit values of MZSFs at non-positive integer points using Stirling numbers}
Let $l_1,\dots,l_r\in\N_0$. We have
\begin{align*}
\zeta^{\star\,\rm rev}_r(-\bfl)
= \sum_{0\le k_i \le l_i \atop 1\le i\le r} 
    \prod_{j=1}^r (-1)^{l_j-k_j}S(l_j,k_j,L_{j-1}+j)
    \frac{(K_j+j-1)!}{(K_{j-1}+j-1)!} \cdot
    \zeta^{\star\,\rm rev}_r(\boldsymbol{0}).
\end{align*}
\end{thm}

\begin{proof}
One can prove this claim using the same strategy as in Theorem \ref{thm: explicit values of MZFs at non-positive integer points using Stirling numbers}. So we leave only an outline here. 
First, prove 
\begin{align*}
 \sum_{0\le k_i \le l_i \atop 1\le i\le r}
 \prod_{j=1}^r s(l_j,k_j,L_{j-1}+j-1) \zeta_r^{\star}(-\bfl',-l_r+s) 
 = \prod_{j=1}^r\frac{(L_j+j-1)!}{(L_{j-1}+j-1)!} \zeta_{r+L_r}(\boldsymbol{0},s). 
\end{align*}
Applying the Stirling polynomial transform \eqref{eqn: Stirling polynomial transform} and putting $s=0$ yields the desired equation.
\end{proof}

By definition, MZF and MZSF do not coincide with each other in their domain of convergence, 
but they do at non-positive integers in the following sense.

\begin{thm}\label{thm: MZF vs MZSF at non-positive integers}
Let $r\in\N$ and $\bfl=(l_1,\dots,l_r)\in\N_0^r$ with $l_1\ge1$. 
We have 
\begin{align}
&\zeta_r^{\rm reg}(-\bfl) = (-1)^{r+|\bfl|} \zeta_r^{\star \,{\rm reg}}(-\bfl), \label{eqn: reg equal reg star}\\
&\zeta_r^{\rm rev}(-\bfl) = (-1)^{r+|\bfl|} \zeta_r^{\star \,{\rm rev}}(-\bfl) \notag
\end{align}
where $|\bfl|:=l_1+\cdots+l_r$.
\end{thm}

\begin{proof}
We can prove in a similar way, so we prove only \eqref{eqn: reg equal reg star}.
We prove the claim by induction on $r$.
For $r=1$, \eqref{eqn: reg equal reg star} holds. 
Note that we have $\zeta(-2l)=0$ and we define $\zeta^{\star}(-l)=\zeta(-l)$ for $l\in\N$.
Assume \eqref{eqn: reg equal reg star} for $r-1$. 
By Lemma \ref{lem: reccurence relation for MZF, regular type} and Lemma \ref{lem: recurrence relation for MZSF}, 
we have 
\begin{align*}
 \zeta_r^{\rm reg}(-\bfl) - (-1)^{r+|\bfl|} \zeta_r^{\star {\rm reg}}(-\bfl) 
 = \frac{-1}{l_r+1} A_{-1}
  + \sum_{k=0}^{l_r}\binom{l_r}{k} A_{k}   
\end{align*}
where 
\begin{align*}
 A_{-1}&:= \zeta_{r-1}^{\rm reg}(-\bfl'',-l_{r-1}-l_r-1) 
          - (-1)^{r+|\bfl|} \zeta_{r-1}^{\star {\rm reg}}(-\bfl'',-l_{r-1}-l_r-1), \\
 A_{k}&:= \zeta_{r-1}^{\rm reg}(-\bfl'',-l_{r-1}-l_r+k)\zeta(-k)
          - (-1)^{r+|\bfl|} \zeta_{r-1}^{\star {\rm reg}}(-\bfl'',-l_{r-1}-l_r+k)\zeta^{\star}(-k).
\end{align*}
By the induction hypothesis, we have 
\begin{align*}
\zeta_{r-1}^{\rm reg}(-\bfl'',-l_{r-1}-l_r-1)
= (-1)^{r-1+|\bfl|-1}\zeta_{r-1}^{\star \rm reg}(-\bfl'',-l_{r-1}-l_r-1).   
\end{align*}
Thus, we get $A_{-1}=0$.

Similarly, by the induction hypothesis gives us 
\begin{align*}
 \zeta_{r-1}^{\rm reg}(-\bfl'',-l_{r-1}-l_r+k)
 = (-1)^{r-1+|\bfl|+k} \zeta_{r-1}^{\star {\rm reg}}(-\bfl'',-l_{r-1}-l_r+k)
\end{align*}
for all $0\le k\le l_r$. 
Hence, we have
\begin{align*}
A_k = (-1)^{r+|\bfl|}\zeta_{r-1}^{\star {\rm reg}}(-\bfl'',-l_{r-1}-l_r+k)\,
      \{(-1)^{k-1}\zeta(-k)-\zeta^{\star}(-k)\}.
\end{align*}
for $k=0,\dots,l_r$.
Let $k\ge1$ and assume that $k\equiv0 \pmod2$. 
We have $\zeta^{\star}(-k)=\zeta(-k)=0$. 
If $k\equiv1 \pmod2$, we have $\zeta^{\star}(-k)=\zeta(-k)$ and $(-1)^{k-1}-1=0$. 
Hence, we have $A_k=0$ for $k\ge1$.
For $k=0$, since $\zeta^{\star}(0)=-\zeta(0)=1/2$, we have $A_0=0$.
Therefore, we get the assertion.
\end{proof}

\begin{rmk}
When $l_1=0$ and $l_2$ is odd (this couple $(0,l_2)$ is a regular point of $\zeta_2(s_1,s_2)$ and $\zeta^{\star}_2(s_1,s_2)$), since
\begin{align*}
 \zeta^{\rm reg}_2(0,-l_2) = -\zeta(-l_2),
\end{align*}
we have 
\begin{align*}
 \zeta^{\star\rm reg}_2(0,-l_2) 
 = \zeta^{\rm reg}_2(0,-l_2) + \zeta(-l_2) = 0 \ne (-1)^{l_2}\zeta^{\rm reg}_2(0,-l_2).
\end{align*}
\end{rmk}

Furthermore, we investigate the general cases of Theorem \ref{thm: MZF vs MZSF at non-positive integers} in the subsequent section.

\subsection{Asymptotic coefficients of MZFs}\label{subsec: asym. coeff.}
For the last part of this paper, we investigate the {\it asymptotic coefficients of MZFs} at any non-positive integers studied in \cite{o}, \cite{s23}, \cite{mo}, and \cite{mmo}.
First, we review Murahara and Onozuka's work. 
In \cite{mo}, they studied the asymptotic behavior of the Hurwitz-Lerch multiple zeta function at non-positive
integer points.
Here, we consider the Hurwitz multiple zeta function defined by
\begin{align*}
 &\zeta_r(\bfs;\bfa)=\zeta_r(s_1,\dots,s_r;a_1,\dots,a_r) \\
 &:= \sum_{m_1,\dots,m_r\ge0}\frac{1}{(m_1+a_1)^{s_1}\cdots (m_1+\cdots+m_r+a_1+\cdots+a_r)^{s_r}}
\end{align*}
where $a_1,\dots,a_r\in\C$ with $\re(a_1)$,$\dots$ ,$\re(a_1+\cdots+a_r)>0$.
Note that when $a_1=\dots=a_r=1$, the Hurwitz MZF $\zeta_r(\bfs;1,\dots,1)$ coincides with the original Euler-Zagier MZF $\zeta_r(\bfs)$. 
We also note that when $a_1=1$, $a_2=\dots=a_r=0$, the Hurwitz MZF $\zeta_r(\bfs;1,0,\dots,0)$ is nothing but the MZSF $\zeta^{\star}_r(\bfs)$.

\begin{thm}(\cite{mo})
Let $r\ge2$, $a_1,\dots,a_r\in\C$ with $\re(a_1)$,$\dots$ ,$\re(a_1+\cdots+a_r)>0$, and $\epsilon_1,\dots,\epsilon_r\in\C$. 
Suppose that $|\epsilon_1|,\dots,|\epsilon_r|$ are sufficiently small with 
$\epsilon_j\ne0$, $\epsilon_j+\dots+\epsilon_r\ne0$ (j=1,\dots,r), and $|\epsilon_k/(\epsilon_j+\dots+\epsilon_r)|\ll1$ as 
$(\epsilon_1,\dots,\epsilon_r)\rightarrow(0,\dots,0)$ for $1\le j\le k\le r$.
Then for $(l_1,\dots,l_r)\in\N_0^r$, we have 
{\small
\begin{equation}\label{eqn: asym. expansion of MZFs}
\begin{split}
 &\zeta_r(-\bfl+\bfepsilon;\bfa)=\zeta_r(-l_1+\epsilon_1,\dots,-l_r+\epsilon_r;a_1,\dots,a_r) \\[1.5mm]
 &= (-1)^{r+|\bfl|}\sum_{\bfd\in I^{r-1}} 
    \sum_{\bfn\in S^{\bfd}(\bfl)}  
    \prod_{j=1}^{n_r} \frac{B_{n_j}(a_j)}{n_j!} f_j(\bfn,\bfl)
    \prod_{1\le j\le r \atop d_j=1}\frac{\epsilon_{j+1}+\cdots+\epsilon_r}{\epsilon_j+\cdots+\epsilon_r}
    + \sum_{j=1}^rO(|\epsilon_j|),
\end{split}
\end{equation}}
where $I:=\{0,1\}$, 
{\small
\begin{align*}
 &f_j(\bfn,\bfl)= f_j(n_1,\dots,n_r,l_1,\dots,l_r) := \left(\sum_{k=1}^j(-n_k+l_k)+j-1\right)_{l_j}, \\
 &S^{\bfd}(\bfl):=\bigcap_{j=1}^{r-1}S^{(d_j)}_{j,r}(\bfl), \quad\text{and} \\
 &S^{(d_j)}_{j,r}(\bfl)
  :=\left\{ \bfn\in\N_0^r \,\Bigg|\, 
            \begin{matrix} n_1+\cdots+n_r=r+l_1+\cdots+l_r \quad\text{and} \\  
                           n_{j+1}+\cdots+n_r\le r-j+l_{j+1}+\cdots+l_r \quad(d_j=0),\quad \text{or} \\
                           n_{j+1}+\cdots+n_r\ge r-j+1+l_j+\cdots+l_r \quad(d_j=1)
            \end{matrix} \right\}.
\end{align*}}
\end{thm}

Now, we define the asymptotic coefficients.

\begin{definition}{(cf. \cite{o}, \cite{s23}, \cite{mmo})}
For $\bfl=(l_1,\dots,l_r)\in\N_0^r$ and $\bfd=(d_1,\dots,d_{r-1})\in\{0,1\}^{r-1}$, 
we define the {\it asymptotic coefficients of MZF} at $(-l_1,\dots,-l_r)$ by
\begin{equation}
 C^{(\bfd)}(-\bfl;\bfa) 
 := (-1)^{r+|\bfl|}\sum_{\bfn\in S^{\bfd}(\bfl)} \prod_{j=1}^{n_r} \frac{B_{n_j}(a_j)}{n_j!} f_j(\bfn,\bfl).
\end{equation}
where $S^{\bfd}(\bfl)$ and $f_j(\bfn,\bfl)$ are defined above. 
When $r=1$, we set $C^{()}(-l;a):=\zeta(-l;a)=-\frac{B_{l+1}(a)}{l+1}$ for our convenience. 
Furthermore, we put $C_{i,r}(-\bfl;\bfa):=C^{(1,\dots,1,0,\dots,0)}(-\bfl;\bfa)$ 
where the number of $1$ is $i-1$ for $1\le i\le r$.
\end{definition}

\begin{definition}{(cf. \cite{o}, \cite{s23}, \cite{mmo})}
For $a_1=\dots=a_r=1$, we put 
\begin{align*}
 C^{(\bfd)}(-\bfl):=C^{(\bfd)}(-\bfl;1,\dots,1).
\end{align*}
For $a_1=1$, $a_2=\dots=a_r=0$, we put 
\begin{align*}
 C^{\star\,(\bfd)}(-\bfl):=C^{(\bfd)}(-\bfl;1,0,\dots,0).
\end{align*}
\end{definition}

In \cite{s23} and \cite{mmo}, authors studied $C^{(\bfd)}(-\bfl)$.

\begin{exa}
When $r=2$, we have 
\begin{align*}
\zeta_2(\epsilon_1,-1+\epsilon_2) = \frac{1}{12} + E, \quad
\zeta_2(-1+\epsilon_1,-1+\epsilon_2) = \frac{1}{360} + \frac{1}{720}\frac{\epsilon_2}{\epsilon_1+\epsilon_2} + E
\end{align*}
where $E:=\sum_{j=1}^2O(|\epsilon_j|)$. So we get 
{\small
\begin{align*}
 C^{(0)}(0,-1)=\frac{1}{12},\ C^{(1)}(0,-1)=0,\quad
 C^{(0)}(-1,-1)=\frac{1}{360},\ C^{(1)}(-1,-1)=\frac{1}{720}.
\end{align*}}
\end{exa}

\begin{rmk}\label{rmk: asymptotic coefficients vs regular/reverse values}
By definition, we have 
\begin{align*}
 &C_{1,r}(-l_1,\dots,-l_r)=C^{(0,\dots,0)}(-\bfl;\boldsymbol{1})=\zeta_r^{{\rm reg}}(-l_1,\dots,-l_r), \\
 &\sum_{\bfd\in I^{r-1}}C^{(\bfd)}(-l_1,\dots,-l_r)=\zeta_r^{{\rm rev}}(-l_1,\dots,-l_r) \quad (\bfl\in\N_0^r).
\end{align*}
\end{rmk}

As shown in \cite{s23}, the asymptotic coefficients have recurrence relations. 

\begin{prop}{(cf. \cite[Theorem 4, 5]{s23})}\label{prop: rec. rel. for asym. coeff.}
For $\bfl=(l_1,\dots,l_r)\in\N_0^r$ and $1\le i<r$, we have 
\begin{align*}
 C_{i,r}(-\bfl;\bfa) 
 = -\sum_{k=0}^{l_r+1}
    \binom{l_r+1}{k}C_{i,r-1}(-\bfl'',-l_{r-1}-k;\bfa')\frac{B_{l_r+1-k}(a_r)}{l_r+1},
\end{align*}
where $\bfl'':=(l_1,\dots,l_{r-2})$ and $\bfa'=(a_1,\dots,a_{r-1})$. Moreover, we have
\begin{align*}
 C_{r,r}(-\bfl;\bfa) 
 = \sum_{k=0}^{l_1+1}
   \binom{l_1+1}{k}\,C_{r-1,r-1}(-l_2-k,-\bfl'';\bfa')\frac{B_{l_1+1-k}(a_1)}{l_1+1}.
\end{align*}
where $\bfl''=(l_3,\dots,l_r)$ and $\bfa'=(a_2,\dots,a_r)$.
\end{prop}

Using these recurrence relations inductively, 
we obtain an explicit formula for $C_{i,r}(-\bfl;\bfa)$ ($\bfl\in\N_0^r$) in terms of Bernoulli polynomials.

\begin{prop} 
Let $r\ge3$, $1\le i\le r-1$, and $\bfl=(l_1,\dots,l_{r})\in\N_0^{r}$. 
We have 
{\small
\begin{equation}\label{eqn: explicit formula for asym. coeff.}
\begin{split}
C_{i,r}(-\bfl;\bfa) 
&= (-1)^{r-i+1}\sum_{k_r=0}^{l_r+1}\sum_{k_{r-1}=0}^{k_r+l_{r-1}+1}
  \cdots\!\!\!\!\!\sum_{k_{i+1}=0}^{k_{i+2}+l_{i+1}+1}\!\! \cdot\ 
  \sum_{k_1=0}^{l_1+1}\sum_{k_2=0}^{k_1+l_2+1}\cdots\!\!\!\!\!\sum_{k_{i-1}=0}^{k_{i-2}+l_{i-1}+1} \\
&\quad \cdot P_1(\bfl_{i-1};\bfk_{i-1};\bfa_{i-1})
             \,C_{1,1}(-l_i-k_{i+1}-k_{i-1};a_i)\,P_2(\bfl^{i+1};\bfk^{i+1};\bfa^{i+1})
\end{split}
\end{equation}}
where $\bfl_{i-1}:=(l_1,\dots,l_{i-1})$, $\bfl^{i+1}:=(l_{i+1},\dots,l_r)$ 
(resp. $\bfk_{i-1}$, $\bfa_{i-1}$ $\bfk^{i+1}$, $\bfa^{i+1}$), and 
\begin{align*}
 P_1(\bfl_{i-1};\bfk_{i-1};\bfa_{i-1}) 
 &:= \prod_{j=1}^{i-1} \binom{k_{j-1}+l_j+1}{k_j} \frac{B_{k_{j-1}+l_j+1-k_j}(a_j)}{k_{j-1}+l_j+1},\\
 P_2(\bfl^{i+1};\bfk^{i+1};\bfa^{i+1}) 
 &:= \prod_{j=i+1}^{r} \binom{k_{j+1}+l_j+1}{k_j} \frac{B_{k_{j+1}+l_j+1-k_j}(a_j)}{k_{j+1}+l_j+1}. 
\end{align*}
Here, $k_0=k_{r+1}:=0$.
\end{prop}

\begin{rmk}
By \eqref{eqn: explicit formula for asym. coeff.}, 
we have 
\begin{align*}
 &C_{1,r}(-\bfl)=C_{1,r}(-l_1,\dots,-l_r;1,\dots,1) \\
 &= (-1)^{|\bfl|-1} \sum_{k_r=0}^{l_r+1}\sum_{k_{r-1}=0}^{k_r+l_{r-1}+1}\cdots \sum_{k_{2}=0}^{k_3+l_{2}+1}
             \binom{l_r+1}{k_r}\binom{k_r+l_{r-1}+1}{k_{r-1}}\cdots \binom{k_3+l_{2}+1}{k_{2}} \\
 &\qquad\qquad \cdot \frac{B_{l_1+k_2+1}}{l_1+k_2+1}\cdots\frac{B_{k_r+l_{r-1}+1-k_{r-1}}}{k_r+l_{r-1}+1}\frac{B_{l_r+1-k_r}}{l_r+1}.
\end{align*}
Note that $B_n(1)=(-1)^nB_n(0)$ and $B_n(0)=B_n$.
Similarly, we have 
\begin{align*}
 &C^{\star}_{1,r}(-\bfl)=C_{1,r}(-l_1,\dots,-l_r;1,0,\dots,0)\\
 &= (-1)^{r-1} \sum_{k_r=0}^{l_r+1}\sum_{k_{r-1}=0}^{k_r+l_{r-1}+1}\cdots \sum_{k_{2}=0}^{k_3+l_{2}+1}
             \binom{l_r+1}{k_r}\binom{k_r+l_{r-1}+1}{k_{r-1}}\cdots \binom{k_3+l_{2}+1}{k_{2}} \\
 &\qquad\qquad \cdot \frac{B_{l_1+k_2+1}}{l_1+k_2+1}\cdots\frac{B_{k_r+l_{r-1}+1-k_{r-1}}}{k_r+l_{r-1}+1}\frac{B_{l_r+1-k_r}}{l_r+1}.
\end{align*}
for $l_1\ge1$ (note that $B_n=(-1)^nB_n$ for $n\ne1$).
Thus, we immediately get \eqref{thm: MZF vs MZSF at non-positive integers}
\begin{align*}
 C_{1,r}(-\bfl)=(-1)^{r+|\bfl|}C^{\star}_{1,r}(-\bfl) \quad(l_1\ge1).
\end{align*}
\end{rmk}

One can see the following, which is a generalization of Theorem \ref{thm: MZF vs MZSF at non-positive integers}.

\begin{prop}
For $2\le i\le r$, $1\le p\le r$ with $p\ne i-1,i+1$, and $\bfl\in\N_0^r$, it follows 
\begin{align*}
 C_{i,r}(-\bfl;0,\dots,0,\overset{p}{1},0,\dots,0) - (-1)^{r+|\bfl|}C_{i,r}(-\bfl;1,\dots,1) 
 = (-1)^{r+|\bfl|}C_{i_p,r-1}(-\bfl'_p;1,\dots,1)
\end{align*} 
where $i_p=\begin{cases} i-1 \quad&(p< i-1) \\ i \quad &(p>i+1)\end{cases}$ and 
\begin{align*}
 -\bfl'_p 
 = \begin{cases}
     (-l_1,\dots, -l_{p-1}-l_p,\dots,-l_r) \quad&(p< i-1) \\
     (-l_1,\dots, -l_p-l_{p+1},\dots,-l_r) \quad &(p>i+1).
   \end{cases}
\end{align*}
Moreover, when $l_i\ge1$, we have 
\begin{align*}
 C_{i,r}(-\bfl;0,\dots,0,\overset{i}{1},0,\dots,0) = (-1)^{r+|\bfl|}C_{i,r}(-\bfl;1,\dots,1).
\end{align*}
\end{prop}

\begin{proof}
Because we can prove in a similar way, we prove only the case $p=1$.
\begin{align*}
 C_{i,r}(-\bfl;1,\boldsymbol{0}) - (-1)^{r+|\bfl|}C_{i,r}(-\bfl;\boldsymbol{1}) 
 = (-1)^{r+|\bfl|}C_{i-1,r-1}(-l_1-l_2,-\bfl'';\boldsymbol{1}).
\end{align*}
The equation \eqref{eqn: explicit formula for asym. coeff.} gives us
\begin{equation*}\begin{split}
 C_{i,r}(-\bfl;1,\dots,1) 
&= (-1)^{r-i+1}\sum_{k_r=0}^{l_r+1}\sum_{k_{r-1}=0}^{k_r+l_{r-1}+1}
  \cdots\!\!\!\!\!\sum_{k_{i+1}=0}^{k_{i+2}+l_{i+1}+1}\!\! \cdot\ 
  \sum_{k_1=0}^{l_1+1}\sum_{k_2=0}^{k_1+l_2+1}\cdots\!\!\!\!\!\sum_{k_{i-1}=0}^{k_{i-2}+l_{i-1}+1} \\
&\quad \cdot P_1(\bfl_{i-1};\bfk_{i-1};\boldsymbol{1})
             \,C_{1,1}(-l_i-k_{i+1}-k_{i-1};1)\,
             P_2(\bfl^{i+1};\bfk^{i+1};\boldsymbol{1})
\end{split}\end{equation*}
where
\begin{align*}
 P_1(\bfl_{i-1};\bfk_{i-1};\boldsymbol{1}) 
 &:= \prod_{j=1}^{i-1} \binom{k_{j-1}+l_j+1}{k_j} \frac{B_{k_{j-1}+l_j+1-k_j}(1)}{k_{j-1}+l_j+1}\\
 P_2(\bfl^{i+1};\bfk^{i+1};\boldsymbol{1}) 
 &:= \prod_{j=i+1}^{r} \binom{k_{j+1}+l_j+1}{k_j} \frac{B_{k_{j+1}+l_j+1-k_j}(1)}{k_{j+1}+l_j+1}.
\end{align*}
Here, $k_0=k_{r+1}:=0$. 
We use $B_n(1)=(-1)^nB_n$ and get 
\begin{equation}\begin{split}\label{eqn: Cir of 1s}
 C_{i,r}(-\bfl;1,\dots,1) 
&= (-1)^{|\bfl|-i}\sum_{k_r=0}^{l_r+1}\sum_{k_{r-1}=0}^{k_r+l_{r-1}+1}
  \cdots\!\!\!\!\!\sum_{k_{i+1}=0}^{k_{i+2}+l_{i+1}+1}\!\! \cdot\ 
  \sum_{k_1=0}^{l_1+1}\sum_{k_2=0}^{k_1+l_2+1}\cdots\!\!\!\!\!\sum_{k_{i-1}=0}^{k_{i-2}+l_{i-1}+1} \\
&\quad \cdot P_1(\bfl_{i-1};\bfk_{i-1};\boldsymbol{0}) \,\frac{B_{l_i+k_{i+1}+k_{i-1}+1}}{l_i+k_{i+1}+k_{i-1}+1}\,
P_2(\bfl^{i+1};\bfk^{i+1};\boldsymbol{0}).
\end{split}\end{equation}
Since $B_n=(-1)^{n+\delta_{n1}}B_n$ where $\delta_{n1}$ is the Kronecker's delta, 
\begin{align*}
 P_1(\bfl_{i-1};\bfk_{i-1};1,\boldsymbol{0}) 
 &= \binom{l_1+1}{k_1} \frac{B_{l_1+1-k_1}(1)}{l_1+1}
     \prod_{j=2}^{i-1} \binom{k_{j-1}+l_j+1}{k_j} \frac{B_{k_{j-1}+l_j+1-k_j}}{k_{j-1}+l_j+1}\\
 &= \binom{l_1+1}{k_1} (-1)^{\delta_{k_1l_1}}\frac{B_{l_1+1-k_1}}{l_1+1} 
     \prod_{j=2}^{i-1} \binom{k_{j-1}+l_j+1}{k_j} \frac{B_{k_{j-1}+l_j+1-k_j}}{k_{j-1}+l_j+1}.
\end{align*}
This is equivalent to
{\small
\begin{align*}
 P_1(\bfl_{i-1};\bfk_{i-1}) 
 = \prod_{j=1}^{i-1} \binom{k_{j-1}+l_j+1}{k_j} \frac{B_{k_{j-1}+l_j+1-k_j}}{k_{j-1}+l_j+1}
    -2\delta_{k_1l_1} 
     \prod_{j=1}^{i-1} \binom{k_{j-1}+l_j+1}{k_j} \frac{B_{k_{j-1}+l_j+1-k_j}}{k_{j-1}+l_j+1} .
\end{align*}}
Thus, by \eqref{eqn: explicit formula for asym. coeff.}, we have 
\begin{equation}\begin{split}\label{eqn: Cir of 1,0s}
 C_{i,r}(-\bfl;1,0,\dots,0) 
&= (-1)^{r-i}\sum_{k_r=0}^{l_r+1}\sum_{k_{r-1}=0}^{k_r+l_{r-1}+1}
  \cdots\!\!\!\!\!\sum_{k_{i+1}=0}^{k_{i+2}+l_{i+1}+1}\!\! \cdot\ 
  \sum_{k_1=0}^{l_1+1}\sum_{k_2=0}^{k_1+l_2+1}\cdots\!\!\!\!\!\sum_{k_{i-1}=0}^{k_{i-2}+l_{i-1}+1} \\
&\quad \cdot P_1(\bfl_{i-1};\bfk_{i-1};\boldsymbol{0}) \,\frac{B_{l_i+k_{i+1}+k_{i-1}+1}}{l_i+k_{i+1}+k_{i-1}+1}\,
             P_2(\bfl^{i+1};\bfk^{i+1};\boldsymbol{0}) \\
&-(-1)^{r-1-i}\sum_{k_r=0}^{l_r+1}\sum_{k_{r-1}=0}^{k_r+l_{r-1}+1}
  \cdots\!\!\!\!\!\sum_{k_{i+1}=0}^{k_{i+2}+l_{i+1}+1}\!\! \cdot\ 
  \sum_{k_2=0}^{l_1+l_2+1}\cdots\!\!\!\!\!\sum_{k_{i-1}=0}^{k_{i-2}+l_{i-1}+1} \\
&\quad \cdot P_1(\bfl'_{i-1};\bfk'_{i-1};\boldsymbol{0})
             \,\frac{B_{l_i+k_{i+1}+k_{i-1}+1}}{l_i+k_{i+1}+k_{i-1}+1}\,
             P_2(\bfl^{i+1};\bfk^{i+1};\boldsymbol{0})
\end{split}\end{equation}
where $\bfl'_{i-1}=(l_1+l_2,l_3,\dots,l_r)$ and $\bfk'_{i-1}=(k_2,\dots,k_{i-1})$.
Therefore, combining \eqref{eqn: Cir of 1s} and \eqref{eqn: Cir of 1,0s}, we obtain
\begin{align*}
 C_{i,r}(-\bfl;1,\boldsymbol{0}) - (-1)^{r+|\bfl|}C_{i,r}(-\bfl;\boldsymbol{1}) 
 = (-1)^{r+|\bfl|}C_{i-1,r-1}(-l_1-l_2,-\bfl'';\boldsymbol{1})
\end{align*}
as desired.
For the second statement, since 
{\small
\begin{align*}
 C_{i,r}(-\bfl;0,\dots,0,\overset{i}{1},0,\dots,0) 
 &= (-1)^{r-i}\sum_{k_r=0}^{l_r+1}\sum_{k_{r-1}=0}^{k_r+l_{r-1}+1}
  \cdots\!\!\!\!\!\sum_{k_{i+1}=0}^{k_{i+2}+l_{i+1}+1}\!\! \cdot\ 
  \sum_{k_1=0}^{l_1+1}\sum_{k_2=0}^{k_1+l_2+1}\cdots\!\!\!\!\!\sum_{k_{i-1}=0}^{k_{i-2}+l_{i-1}+1} \\
&\quad \cdot P_1(\bfl_{i-1};\bfk_{i-1};\boldsymbol{0}) \,\frac{B_{l_i+k_{i+1}+k_{i-1}+1}}{l_i+k_{i+1}+k_{i-1}+1}\,
             P_2(\bfl^{i+1};\bfk^{i+1};\boldsymbol{0}),
\end{align*}}
we get the claim.
\end{proof}

\begin{rmk}
Since $B_n(1-a)= (-1)^nB_n(a)$ ($0\le a\le 1$) in general, we have 
\begin{align*}
 C_{i,r}(-\bfl;\bfa) = (-1)^{|\bfl|-i}C_{i,r}(-\bfl;\boldsymbol{1}-\bfa)
\end{align*}
where $\boldsymbol{1}-\bfa:=(1-a_1,\dots,1-a_r)$.
\end{rmk}

In \cite{mmo}, Matsusaka, Muarahara, and Onozuka studied the asymptotic coefficients of MZFs at the origin.
More specifically, they revealed the connection between $C_{i,r}(\boldsymbol{0})=C_{i,r}(0,\dots,0;1,\dots,1)$ and the {\it Gregory coefficients} $G_{m,n}$ ($m,n\in\N_0$) defined by the following generating series 
\begin{equation}\label{eqn: definition of the gen. Gregory coeff.}
 \mathcal{G}(x,y):=\sum_{m,n\in\N_0}G_{m,n}x^my^n
 := \frac{y\log^2(1+x) - x\log^2(1+y)}{\log(1+x) - \log(1+y)}
\end{equation}
where $\log^2(z):=(\log(z))^2$.

\begin{thm}(\cite{mmo}) \label{thm: asym. coeff. vs gen. Gregory coeff.}
For $1\le i\le r$, we have $C_{i,r}(\boldsymbol{0})=G_{i,r-i+2}$.
\end{thm}

Then, it is natural to ask for other cases, that is, 
possible connections with $C^{(\bfd)}(-\bfl)$ ($\bfl\in\N_0^r$) and the generalized Gregory coefficients. 
Here, we give a partial answer.

\begin{thm}\label{thm: reverse values vs Gregory coefficients}
Let $\bfl=(l_1,\dots,l_{r})\in\N_0^{r}$.
It follows that  
\begin{equation}\label{eqn: reverse val. at any non-positive int. vs Gregory coeff.}
 \zeta_{r}^{\rm rev}(-\bfl) 
 = \sum_{0\le k_i \le l_i \atop 1\le i\le r} 
    c(\bfk,\bfl)
   \sum_{j=0}^{\lfloor\frac{r-1}{2}\rfloor}
   \sum_{k=0}^{r-1-2j}\sum_{(\bfm,\bfn)\in J(j,k)}
   \prod_{p=1}^{j+1}G_{m_p,n_p-m_p+2}
\end{equation}
where the symbol $\lfloor x\rfloor$ stands for the integer parts of $x\in\R$, 
\begin{align*}
c(\bfk,\bfl) := \prod_{j=1}^r S(l_j,k_j,L_{j-1}+j) \frac{(K_j+j-1)!}{(K_{j-1}+j-1)!},
\end{align*}
and 
\begin{align*}
 J(j,k) := \left\{ (\bfm,\bfn)\in \N^{2(j+1)}\,\Big|\, 
                   \sum_{p=1}^{j+1}m_p = 2j+1+k,\, \sum_{p=1}^{j+1}n_p = r,\, m_p\le n_p  \right\}.
\end{align*}
\end{thm}

To prove Theorem \ref{thm: reverse values vs Gregory coefficients}, we prepare the following lemmas.

\begin{lem}\label{lem: sum of asym. coeff. of non-primitive cases}
Let $j, k\in\N_0$ be such that $0\le j\le \lfloor \frac{r-1}{2}\rfloor$ and $0\le k\le r-1-2j$.
It follows 
\begin{equation}
 \sum_{(d_1,\dots,d_{r-1})\in I_r(j,k)}C^{(d_1,\dots,d_{r-1})}(\boldsymbol{0})
 = \sum_{ (\bfm,\bfn)\in J(j,k) }
   \prod_{p=1}^{j+1} G_{m_p,n_p-m_p+2}
\end{equation}
where 
\begin{align*}
 I_r(j,k) 
 := \left\{ (d_1,\dots,d_{r-1})\in I^{r-1} \,\Bigg|\, 
      \begin{matrix} \#\{1\le l\le r-2\,|\,(d_l,d_{l+1})=(0,1)\}=j, \\ 
                     \#\{1\le m\le r-1\,|\,d_{m}=1,d_{m-1}\ne0\}=k
      \end{matrix} \right\}.
\end{align*}
\end{lem}

\begin{proof}
Since $(d_1,\dots,d_{r-1})\in I_{r}(j,k)$, we can write 
\begin{align*}
 (d_1,\dots,d_{r-1})
 = (d_1,\dots,d_{i_1-1},0,1,d_{i_1+2},\dots,d_{i_2-1},0,1,\dots,0,1,d_{i_j+2},\dots,d_{r-1})
\end{align*}
for $1\le i_1<i_2<\cdots<i_j<r$ and note that for any $0\le p\le j$, 
\begin{equation}\label{eqn: decomposed elements are primitive}
 (d_{i_p+1},\dots,d_{i_{p+1}-1}) = (1,\dots,1,0,\dots,0).
\end{equation}
We consider $i_0=0$ and $i_{j+1}=r$. 
We also remark that $(d_{i_p+2},\dots,d_{i_{p+1}-1})$ can be empty. 
Thus, by the following Lemma \ref{lem: red. formula for asym. coeff. of MZF}, we obtain 
\begin{align*}
C^{(d_1,\dots,d_{r-1})}(\boldsymbol{0})
= \prod_{p=0}^{j} C^{(d_{i_{p}+1},\dots,d_{i_{p+1}-1})}(\boldsymbol{0})
\end{align*}
for $(d_1,\dots,d_{r-1})\in I_{r}(j,k)$.
Since \eqref{eqn: decomposed elements are primitive} for all $0\le p\le j$,
we have 
\begin{align*}
 \prod_{p=0}^{j} C^{(d_{i_{p}+1},\dots,d_{i_{p+1}-1})}(\boldsymbol{0})
 = \prod_{p=1}^{j+1} C_{m_p,n_p}(\boldsymbol{0})
\end{align*}
where $m_p,n_p\in\N$ with
\begin{align*}
\sum_{p=1}^{j+1}m_p = 2j+1+k,\quad \sum_{p=1}^{j+1}n_p = r, \quad m_p\le n_p.
\end{align*}
Therefore, by Theorem \ref{thm: asym. coeff. vs gen. Gregory coeff.}, we get the claim.
\end{proof}

\begin{lem}(\cite{mmo})\label{lem: red. formula for asym. coeff. of MZF} 
For any $0<k<r-1$ and, we have 
\begin{equation*}
 C^{(d_1,\dots,d_{k-1},0,1,d_{k+2},\dots,d_{r-1})}(\boldsymbol{0})
 =C^{(d_1,\dots,d_{k-1})}(\boldsymbol{0})C^{(1,d_{k+2},\dots,d_{r-1})}(\boldsymbol{0}). 
\end{equation*}
\end{lem}

We note that Sasaki proved such a decomposition formula for general cases, see \cite[Theorem 6]{s23}.

\begin{lem}\label{lem: reverse values at the origin vs gen. Gregory coefficients}
It follows  
\begin{align*}
 \zeta_r^{{\rm rev}}(\boldsymbol{0}) 
 = \sum_{j=0}^{\lfloor\frac{r-1}{2}\rfloor}
   \sum_{k=0}^{r-1-2j}\sum_{(\bfm,\bfn)\in J(j,k)}
   \prod_{p=1}^{j+1}G_{m_p,n_p-m_p+2}.
\end{align*}
\end{lem}

\begin{proof}
We use the following set-theoretical decomposition:
\begin{align*}
 I^{r-1} = \{0,1\}^{r-1}= \bigcup_{j=0}^{\lfloor\frac{r-1}{2}\rfloor}\bigcup_{k=0}^{r-1-2j} I_r(j,k).
\end{align*}
This decomposition yields 
\begin{align*}
 \zeta_r^{{\rm rev}}(\boldsymbol{0}) = \sum_{\bfd\in I^{r-1}}C^{(\bfd)}(\boldsymbol{0}) 
 = \sum_{j=0}^{\lfloor\frac{r-1}{2}\rfloor}
   \sum_{k=0}^{r-1-2j}\sum_{\bfd\in I_r(j,k)}C^{(d_1,\dots,d_{r-1})}(\boldsymbol{0}) .
\end{align*}
By Lemma \ref{lem: sum of asym. coeff. of non-primitive cases}, we obtain the claim.
\end{proof}

We finish the present paper with the proof of Theorem \ref{thm: reverse values vs Gregory coefficients}.

\bigskip\noindent
{\it Proof of Theorem \ref{thm: reverse values vs Gregory coefficients}.} 
By Corollary \ref{cor: applying Stirling transform}, we have 
\begin{align*}
 \zeta_r^{{\rm rev}}(-\bfl) 
 = \sum_{0\le k_i \le l_i \atop 1\le i\le r} 
    c(\bfk,\bfl)\, \zeta_r^{{\rm rev}}(\boldsymbol{0})
\end{align*}
for any $\bfl\in\N_0^r$.
Hence, by Lemma \ref{lem: reverse values at the origin vs gen. Gregory coefficients}, 
we obtain \eqref{eqn: reverse val. at any non-positive int. vs Gregory coeff.}.
\qed


\end{document}